\documentclass[11pt]{article}%
\usepackage{amsmath}
\usepackage{amsfonts}
\usepackage{amssymb}
\usepackage{graphicx}%
\newtheorem{theorem}{Theorem}

\newtheorem{corollary}[theorem]{Corollary}

\newtheorem{definition}[theorem]{Definition}
\newtheorem{example}[theorem]{Example}

\newtheorem{lemma}[theorem]{Lemma}

\newtheorem{remark}[theorem]{Remark}

\newenvironment{proof}[1][Proof]{\noindent\textbf{#1.} }{\ \rule{0.5em}{0.5em}}
\begin{document}

\centerline{{\Large Semiconjugate Factorizations of Higher Order }}
\vspace{0.5ex} \centerline{{\Large Linear Difference Equations in Rings}}

\vspace{1ex}

\centerline{H. SEDAGHAT \footnote{Department of Mathematics, Virginia Commonwealth University Richmond, Virginia, 23284-2014, USA; Email: hsedagha@vcu.edu}}

\vspace{2ex}

\begin{abstract}
We study linear difference equations with variable coefficients in a ring using
a new nonlinear method. In
a ring with identity, if the homogeneous part of the linear equation has a
solution in the unit group of the ring (i.e., a unitary solution) then we show that
the equation decomposes into two linear equations of lower orders. This
decomposition, known as a semiconjugate factorization in the nonlinear theory,
generalizes the classical operator factorization in the linear context.
Sequences of ratios of consecutive terms of a unitary solution are used to
obtain the semiconjugate factorization. Such sequences, known as
eigensequences are well-suited to variable coefficients; for instance, they
provide a natural context for the expression of the classical
Poincar\'{e}-Perron Theorem. We discuss some applications to linear
difference equations with periodic coefficients and also derive formulas for the
general solutions of linear functional recurrences satisfied by the classical
special functions such as the modified Bessel and Chebyshev.

\end{abstract}

\bigskip

\section{Introduction}

A linear, non-homogeneous difference equation with variable coefficients is
defined as
\begin{equation}
x_{n+1}=a_{0,n}x_{n}+a_{1,n}x_{n-1}+\cdots+a_{k,n}x_{n-k}+b_{n} \label{Lde}%
\end{equation}
where $\{b_{n}\}$ and $\{a_{j,n}\}$ are given sequences in a nontrivial ring
$R$ for $j=0,1,\ldots,k.$

Equation (\ref{Lde}) may be unfolded in the standard way to a map of a
$(k+1)$-dimensional $R$-module (or vector space, if $R$ is a field). In this
form there is substantial published research that extends the classical theory
on the real line to a diverse selection of rings that include finite rings,
rings of polynomials and other rings of functions as well as abstract rings of
various types. See, e.g., \cite{abu}, \cite{bir}, \cite{bro}, \cite{hen},
\cite{kur}, \cite{lak}, \cite{rob}, \cite{sen}.

We pursue a different approach in this paper, using a method originally
developed for nonlinear difference equations. Specifically, we explore the
existence of a \textit{semiconjugate factorization} of (\ref{Lde}) in its
underlying ring $R$ into two or more linear difference equations of lower
orders. While a linear equation may have many different semiconjugate
factorizations, we seek one where the equations of lower orders are also
linear. The key requirement for the existence of a semiconjugate factorization
is the existence of a special type of sequence that we call an
\textit{eigensequence} of the homogeneous part of (\ref{Lde}); eigenvalues are
essentially constant eigensequences.

We show that in a ring $R$ with identity, quotients of the consecutive terms
of a unitary (invertible) solution form an eigensequence in $R$. So it is not
surprising that the classical \textit{Poincar\'{e}-Perron Theorem } finds a
simple and natural expression in this context; see Section \ref{ppt} below.
Semiconjugate factorization generalizes the classical notion of operator
factorization of a homogeneous equation with constant coefficients over the
real or complex numbers. More generally, semiconjugate factorizations of
linear equations over arbitrary nontrivial fields have been studied in Chapter
7 of \cite{fsor}. The results in this paper substantially extend those in
\cite{fsor} to rings and add several new results, including results for
equations with periodic coefficients and for recurrences in rings of
functions; see Sections \ref{sper}-\ref{func} below.

Establishing the existence of an eigensequence and calculating it are problems
that are more challenging in rings than in fields but the rewards are also
potentially greater. In particular, rings of functions on a given set
typically are not fields but the semiconjugate factorization method applies to
linear difference equations in such rings in essentially the same way (finding
an eigensequence) as it does to linear difference equations in a field. We
study semiconjugate factorizations of linear difference equations in rings of
functions and apply the results to obtain explicit formulas for the solutions
of some known higher order functional difference equations; e.g., second-order
recurrences that define modified Bessel functions or Chebyshev polynomials,
but with arbitrary initial functions.

\section{Preliminaries\label{scfg}}

A (forward) \textit{solution} of (\ref{Lde}) is defined, as usual, to be any
sequence $\{x_{n}\}$ in $R$ that satisfies the equation for $n\geq0$. Given
the recursive nature of (\ref{Lde}) it is clear that with any $k+1$ given
initial values $x_{j}\in R$ for $j=0,1,\ldots,k$ (the number $k+1$ being the
order of the difference equation) (\ref{Lde}) generates a unique solution in
$R$ through iteration, since $R$ is closed under its addition and
multiplication. If $b_{n}=0$ for all $n$ then (\ref{Lde}) is
\textit{homogeneous} and in this case, the constant sequence $x_{n}=0$ for
$n\geq0$ is a solution of (\ref{Lde}), namely, the trivial solution.

In the classical theory of higher order linear difference equations in the
field of real numbers, linear operators such as the ones used to define
(\ref{Lde}) may be \textquotedblleft factored\textquotedblright\ using their
eigenvalues; see, e.g., Section 2.3 in \cite{elay}. This elementary procedure
yields both a reduction of order for a linear difference equation and a
symbolic \textquotedblleft operator method\textquotedblright\ for obtaining
its solutions. For discussions of these basic classical notions, including
operator methods, see \cite{elay} or \cite{jor}.

In this section, for the reader's convenience we present some general results
from \cite{fsor} that are valid for all difference equations of recursive
type, not just the linear ones.

Let $\mathcal{G}$ be a nontrivial group and consider the recursive difference
equation, or recurrence%
\begin{equation}
x_{n+1}=f_{n}(x_{n},x_{n-1},\ldots,x_{n-k}) \label{dek}%
\end{equation}
where $f_{n}:\mathcal{G}^{k+1}\rightarrow\mathcal{G}$ is a given function for
each $n\geq0.$ Starting from a set of $k+1$ initial values $x_{j}%
\in\mathcal{G}$ for $j=0,1,\ldots,k$ a unique solution of (\ref{dek}) is
obtained by iteration.

Equation (\ref{dek}) may be unfolded in the usual way to a first-order
recurrence
\[
X_{n+1}=\mathfrak{F}_{n}(X_{n})
\]
on $\mathcal{G}^{k+1}$ where $\mathfrak{F}_{n}:\mathcal{G}^{k+1}%
\rightarrow\mathcal{G}^{k+1}$. Let $k\geq1$\textit{, }$1\leq m\leq k.$ Suppose
that there is a sequence of maps $\Phi_{n}:\mathcal{G}^{m}\rightarrow
\mathcal{G}^{m}$ and a sequence of surjective maps $H_{n}:\mathcal{G}%
^{k+1}\rightarrow\mathcal{G}^{m}$ that satisfy the \textit{semiconjugate
relation} $H_{n+1}\circ\mathfrak{F}_{n}=\Phi_{n}\circ H_{n}$ for a given pair
of function sequences $\{\mathfrak{F}_{n}\}$ and $\{\Phi_{n}\}.$ Then we say
that $\mathfrak{F}_{n}$ is \textit{semiconjugate} to $\Phi_{n}$ for each $n$
and that the sequence $\{H_{n}\}$ is a \textit{form symmetry} of (\ref{dek}).
Since $m<k+1,$ the form symmetry $\{H_{n}\}$ is \textit{order-reducing}.

We state the next core result from \cite{fsor} as a lemma here.

\begin{lemma}
\label{scf}(A semiconjugate factorization theorem) Let $\mathcal{G}$ be a
nontrivial group\textit{ and let }$k\geq1$\textit{, }$1\leq m\leq k$\textit{
be integers. If }$h_{n}:\mathcal{G}^{k-m+1}\rightarrow\mathcal{G}$\textit{\ is
a sequence of functions }and the functions $H_{n}:\mathcal{G}^{k+1}%
\rightarrow\mathcal{G}^{m}$ are defined by\textit{\ }%

\[
H_{n}(u_{0},\ldots,u_{k})=[u_{0}\ast h_{n}(u_{1},\ldots,u_{k+1-m}%
),\ldots,u_{m-1}\ast h_{n-m+1}(u_{m},\ldots,u_{k})]
\]
where $\ast$ denotes the group operation in $\mathcal{G}$ then the following
statements are true:\textit{\ }

(a) The function $H_{n}$ is surjective for every $n\geq0$.

(b) \textit{If }$\{H_{n}\}$\textit{\ is an order-reducing form symmetry then
the difference equation (\ref{dek})\ is equivalent to the system of equations
}%
\begin{align}
t_{n+1}  &  =\phi_{n}(t_{n},\ldots,t_{n-m+1}),\label{tdf}\\
x_{n+1}  &  =t_{n+1}\ast h_{n+1}(x_{n},\ldots,x_{n-k+m})^{-1} \label{tdcf}%
\end{align}

\textit{whose orders }$m$\textit{\ and }$k+1-m$\textit{\ respectively, add up
to the order of (\ref{dek}). }

(c) The map $\Phi_{n}:\mathcal{G}^{m}\rightarrow\mathcal{G}^{m}$ is the
standard unfolding of Eq.(\ref{tdf}) for each $n\geq0$.
\end{lemma}

\begin{remark}
\label{scj}Part (c) above permits us to stay within the context of higher
order difference equations. Being able to work within this context is
especially beneficial in the case of \textbf{nonrecursive} equations
(including some linear ones) which do not in general unfold to maps on modules
over rings (or vector spaces over fields) and therefore, their solutions are
not determined via group actions; see the final section of this paper for an
illustrative example. Extensions of the method of this paper to nonrecursive
equations no longer rely on semiconjugacy but they do retain the basic
concepts of form symmetry and factor-cofactor pairs; see \cite{fsor}, Chapter 8.
\end{remark}

\begin{definition}
The pair of equations (\ref{tdf}) and (\ref{tdcf}) constitute the
\textbf{semiconjugate factorization}, or \textbf{sc-factorization} of
(\ref{dek}). This pair of equations is a triangular system (see \cite{sj})
since (\ref{tdf}) is independent of (\ref{tdcf}). We call (\ref{tdf}) the
\textbf{factor }equation of (\ref{dek}) and (\ref{tdcf}) its \textbf{cofactor
}equation\textit{. }
\end{definition}

Note that (\ref{tdf}) has order $m$ and (\ref{tdcf}) has order $k+1-m.$
Consider the following special case of $H_{n}$ in Lemma \ref{scf} with $m=k$
\begin{equation}
H_{n}(u_{0},u_{1},\ldots,u_{k})=[u_{0}\ast h_{n}(u_{1}),u_{1}\ast
h_{n-1}(u_{2}),\ldots,u_{k-1}\ast h_{n-k+1}(u_{k})] \label{fsk1}%
\end{equation}
where $h_{n}:\mathcal{G}\rightarrow\mathcal{G}$ being a given sequence of
maps. The semiconjugate factorization of (\ref{dek}) in this case is%
\begin{align}
t_{n+1}  &  =\phi_{n}(t_{n},\ldots,t_{n-k+1}),\label{tdf1}\\
x_{n+1}  &  =t_{n+1}\ast h_{n+1}(x_{n})^{-1} \label{tdcf1}%
\end{align}
in which the factor equation has order $k$ and the cofactor equation has order 1.

The next result gives a necessary and sufficient condition for the existence
of a form symmetry of type (\ref{fsk1}); see \cite{fsor} for the proof.

\begin{lemma}
(Invertible-map criterion) Let $\mathcal{G}$ be a nontrivial group\textit{ and
}assume that $h_{n}:\mathcal{G}\rightarrow\mathcal{G}$ is a sequence of
bijections. For arbitrary elements $u_{0},v_{1},\ldots,v_{k}\in\mathcal{G}$
and every $n\geq0$ define $\zeta_{0,n}(u_{0})\equiv u_{0}$ and for
$j=1,\ldots,k$ define%
\begin{equation}
\zeta_{j,n}(u_{0},v_{1},\ldots,v_{j})=h_{n-j+1}^{-1}(\zeta_{j-1,n}(u_{0}%
,v_{1},\ldots,v_{j-1})^{-1}\ast v_{j}). \label{hzetajn}%
\end{equation}

Then (\ref{dek}) has the form symmetry (\ref{fsk1}) if and only if the
quantity
\begin{equation}
f_{n}(\zeta_{0,n},\zeta_{1,n}(u_{0},v_{1}),\ldots,\zeta_{k,n}(u_{0}%
,v_{1},\ldots,v_{k}))\ast h_{n+1}(u_{0}) \label{tdhinvcrit}%
\end{equation}

is independent of $u_{0}$ for every $n\geq0$. In this case (\ref{dek}) has a
semiconjugate factorization into (\ref{tdf1}) and (\ref{tdcf1}) where the
factor functions in (\ref{tdf1}) are given by%
\begin{equation}
\phi_{n}(v_{1},\ldots,v_{k})=f_{n}(\zeta_{0,n},\zeta_{1,n}(u_{0},v_{1}%
),\ldots,\zeta_{k,n}(u_{0},v_{1},\ldots,v_{k}))\ast h_{n+1}(u_{0}).
\label{tdinvf}%
\end{equation}

\end{lemma}

In the context of rings, the group $\mathcal{G}$ in the preceding result is
the additive group of the ring so that $\ast$ denotes addition and thus
(\ref{hzetajn}), (\ref{tdhinvcrit}) and (\ref{tdcf1}) read, respectively, as
follows%
\begin{gather*}
\zeta_{j,n}(u_{0},v_{1},\ldots,v_{j})=h_{n-j+1}^{-1}(v_{j}-\zeta_{j-1,n}%
(u_{0},v_{1},\ldots,v_{j-1})),\\
f_{n}(\zeta_{0,n},\zeta_{1,n}(u_{0},v_{1}),\ldots,\zeta_{k,n}(u_{0}%
,v_{1},\ldots,v_{k}))+h_{n+1}(u_{0}),\text{ and:}\\
x_{n+1}=t_{n+1}-h_{n+1}(x_{n}).
\end{gather*}

A basic class of maps $h_{n}$ in rings is defined next.

\begin{definition}
\label{fsl}Let $R$ be a ring and let $\{\alpha_{n}\}$ be a sequence in $R$
such that $\alpha_{n}\not =0$ for all $n.$ A\textit{ }\textbf{linear form
symmetry} is defined as the special case of (\ref{fsk1}) with $h_{n}%
(u)=-\alpha_{n}u;$ i.e.,
\begin{equation}
\lbrack u_{0}-\alpha_{n}u_{1},u_{1}-\alpha_{n-1}u_{2},\ldots,u_{k-1}%
-\alpha_{n-k+1}u_{k}] \label{lfs}%
\end{equation}

\end{definition}

If $\alpha$ is \textit{not} a zero divisor then $h(u)=-\alpha u$ is one-to-one
or injective since for every $u,v\in R$
\[
h(u)=h(v)\Rightarrow-\alpha(u-v)=0\Rightarrow u-v=0.
\]
In general, $h$ is not surjective even if $R$ contains no zero divisors
(consider $\alpha\in\mathbb{Z}$, $\alpha\not =\pm1$). But if $R$ has an
identity and each $\alpha$ is a unit then each $h$ is a bijection with inverse
$h^{-1}(u)=-\alpha^{-1}u.$

\begin{remark}
Many \textit{nonlinear} difference equations possess the linear form symmetry
(\ref{lfs}); see \cite{sedb} and \cite{fsor}.
\end{remark}

\section{SC-factorization in rings\label{scfr}}

Assume that the underlying ring $R$ of (\ref{Lde}) has a (multiplicative)
identity denoted by 1. For such a ring, the set of all units (elements having
multiplicative inverses, or reciprocals) is a group, namely the \textit{unit
group}, that we denote by $G.$

A \textit{unitary sequence} is any sequence in $G.$ If $\{u_{n}\}$ is a
unitary sequence then the sequence $\{u_{n+1}u_{n}^{-1}\}$ of \textit{right
ratios} of $\{u_{n}\}$ is well-defined and unitary. Similarly, the sequence
$\{u_{n}^{-1}u_{n+1}\}$ of \textit{left ratios }is well-defined and unitary.
If $R$ is commutative then the sequences $\{u_{n+1}u_{n}^{-1}\}$ and
$\{u_{n}^{-1}u_{n+1}\}$ are the same, representing the\textit{ ratios
sequence} of $\{u_{n}\}.$

Call two sequences $\{x_{n}\}$ and $\{y_{n}\}$ in $R$ \textit{right
equivalent} (or \textit{left equivalent}) if there is a unit $u$ such that
$y_{n}=x_{n}u$ (or $y_{n}=ux_{n}$) for all $n$. These two relations on the set
of sequences in $R$ are indeed equivalence relations, and if $\{x_{n}\}$ is
unitary then so are $\{x_{n}u\}$ and $\{ux_{n}\}$.

The next result has the same flavor as the result in calculus which states
that differentiable functions having the same derivative are equal up to a constant.

\begin{lemma}
\label{rls}Let $R$ have an identity and $\{x_{n}\}$ and $\{y_{n}\}$ be unitary
sequences. Then $\{x_{n}\}$ and $\{y_{n}\}$ are right (or left) equivalent if
and only if their sequences of right (or left) ratios are equal.
\end{lemma}

\begin{proof}
Suppose that $\{x_{n}\}$ and $\{y_{n}\}$ are right equivalent. Then
$y_{n}=x_{n}u$ for some unit $u$ and all $n$ so that%
\[
y_{n+1}y_{n}^{-1}=x_{n+1}u(x_{n}u)^{-1}=x_{n+1}u(u^{-1}x_{n}^{-1}%
)=x_{n+1}x_{n}^{-1},
\]
i.e., the sequences of right ratios are the same. Conversely, suppose that
$y_{n+1}y_{n}^{-1}=x_{n+1}x_{n}^{-1}$ for all $n\geq0$ and define
$u=x_{0}^{-1}y_{0}.$ Then
\[
x_{1}u=x_{1}x_{0}^{-1}y_{0}=y_{1}y_{0}^{-1}y_{0}=y_{1}.
\]

This equality also implies that $u=x_{1}^{-1}y_{1}$ so the preceding argument
may be repeated to show that $x_{n}u=y_{n}$ for all $n\geq0.$ A similar
argument proves the left-handed case.
\end{proof}

Recall that the existence of a linear form symmetry (\ref{lfs}) implies that
(\ref{dek}) has a semiconjugate factorization with a first-order, linear
non-homogeneous cofactor equation%
\begin{equation}
x_{n+1}=t_{n+1}+\alpha_{n+1}x_{n}. \label{tdfslc}%
\end{equation}

The following necessary and sufficient condition for the existence of a linear
form symmetry is a consequence of the invertible-map criterion with
$\mathcal{G}$ being the additive group of the ring.

\begin{lemma}
\label{tdlfs} Let $R$ be a ring with identity. Equation (\ref{dek}) has the
linear form symmetry (\ref{lfs}) if and only if there is a unitary sequence
$\{\alpha_{n}\}$ in $R$ such that the quantity%
\begin{equation}
f_{n}(u_{0},\zeta_{1,n}(u_{0},v_{1}),\ldots,\zeta_{k,n}(u_{0},v_{1}%
,\ldots,v_{k}))-\alpha_{n+1}u_{0} \label{tdlfscr}%
\end{equation}
is independent of $u_{0}$ for all $n$, where for $j=1,\ldots,k$,
\begin{equation}
\zeta_{j,n}(u_{0},v_{1},\ldots,v_{j})=\left(  \prod_{i=0}^{j-1}\alpha
_{n-i}\right)  ^{-1}u_{0}-\sum_{i=1}^{j}\left(
{\displaystyle\prod_{m=i}^{j}}
\alpha_{n-m+1}\right)  ^{-1}v_{i}. \label{zen}%
\end{equation}

\end{lemma}

\begin{proof}
Define $h_{n}(u)=-\alpha_{n}u$ for each $n$ so that $h_{n}^{-1}(u)=-\alpha
_{n}^{-1}u$ for all $n.$ If we define $\zeta_{j,0}=u_{0}$ and for
$j=1,\ldots,k$ set
\[
\zeta_{j,n}(u_{0},v_{1},\ldots,v_{j})=\alpha_{n-j+1}^{-1}[\zeta_{j-1,n}%
(u_{0},v_{1},\ldots,v_{j-1})-v_{j}].
\]
recursively then the first assertion of the lemma is true by the
invertible-map criterion. To prove (\ref{zen}), observe that%
\[
\zeta_{1,n}(u_{0},v_{1})=\alpha_{n}^{-1}(u_{0}-v_{1})=\alpha_{n}^{-1}%
u_{0}-\alpha_{n}^{-1}v_{1}%
\]
which proves (\ref{zen}) if $j=1$. Suppose that (\ref{zen}) is true for
indices less than $j$ where $j\leq k.$ Then
%

\begin{align*}
\zeta_{j,n}(u_{0},v_{1},\ldots,v_{j})  &  =\alpha_{n-j+1}^{-1}[\zeta
_{j-1,n}(u_{0},v_{1},\ldots,v_{j-1})-v_{j}]\\
&  =\alpha_{n-j+1}^{-1}\left[  \left(  \prod_{i=0}^{j-2}\alpha_{n-i}\right)
^{-1}u_{0}-\sum_{i=1}^{j-1}\left(
{\displaystyle\prod_{m=i}^{j-1}}
\alpha_{n-m+1}\right)  ^{-1}v_{i}-v_{j}\right] \\
&  =\left(  \prod_{i=0}^{j-1}\alpha_{n-i}\right)  ^{-1}u_{0}-\sum_{i=1}%
^{j-1}\left(
{\displaystyle\prod_{m=i}^{j}}
\alpha_{n-m+1}\right)  ^{-1}v_{i}-\alpha_{n-j+1}^{-1}v_{j}\\
&  =\left(  \prod_{i=0}^{j-1}\alpha_{n-i}\right)  ^{-1}u_{0}-\sum_{i=1}%
^{j}\left(
{\displaystyle\prod_{m=i}^{j}}
\alpha_{n-m+1}\right)  ^{-1}v_{i}%
\end{align*}
and the proof is complete.
\end{proof}

The following application of Lemma \ref{tdlfs} gives a sc-factorization
theorem for linear difference equations in a ring with identity; also see the
comments following the theorem.

\begin{theorem}
\label{sf1}Let $R$ be a ring with identity. The linear equation (\ref{Lde})
has the linear form symmetry (\ref{lfs}) with unit coefficients if there is a
unitary sequence $\{\alpha_{n}\}$ that satisfies the relation
\begin{equation}
\alpha_{n+1}=a_{0,n}+\sum_{j=1}^{k}a_{j,n}\left(  \prod_{i=0}^{j-1}%
\alpha_{n-i}\right)  ^{-1}. \label{ricn}%
\end{equation}

The corresponding semiconjugate factorization of (\ref{Lde}) is%
\begin{align}
t_{n+1}  &  =a_{0,n}^{\prime}t_{n}+a_{1,n}^{\prime}t_{n-1}+\cdots
+a_{k-1,n}^{\prime}t_{n-k+1}+b_{n}\label{genlinf}\\
x_{n+1}  &  =\alpha_{n+1}x_{n}+t_{n+1} \label{genlincf}%
\end{align}

where for $m=0,\ldots,k-1$, $t_{m+1}=x_{m+1}-\alpha_{m+1}x_{m}$ and
\[
a_{m,n}^{\prime}=-\sum_{i=m+1}^{k}a_{i,n}\left(  \prod_{j=m+1}^{i}%
\alpha_{n-j+1}\right)  ^{-1}.
\]

\end{theorem}

\begin{proof}
By Lemma \ref{tdlfs} it is only necessary to determine a unitary sequence
$\{\alpha_{n}\}$ in $R$ such that for each $n$ (\ref{tdlfscr}) is independent
of $u_{0}$ for the following functions%
\[
f_{n}(u_{0},\ldots,u_{k})=a_{0,n}u_{0}+a_{1,n}u_{1}+\cdots+a_{k,n}u_{k}%
+b_{n}.
\]

For arbitrary $u_{0},v_{1},\ldots,v_{k}\in R$ and $j=1,\ldots,k$ define
$\zeta_{j,n}(u_{0},v_{1},\ldots,v_{j})$ as in Lemma \ref{tdlfs}. Then
expression (\ref{tdlfscr}) is\smallskip%

\begin{gather}
-\alpha_{n+1}u_{0}+b_{n}+a_{0,n}u_{0}+a_{1,n}\zeta_{1,n}(u_{0},v_{1}%
)+\cdots+a_{k,n}\zeta_{k,n}(u_{0},v_{1},\ldots,v_{k})=\nonumber\\
b_{n}+\left[  -\alpha_{n+1}+a_{0,n}+\sum_{j=1}^{k}a_{j,n}\left(  \prod
_{i=0}^{j-1}\alpha_{n-i}\right)  ^{-1}\right]  u_{0}-\sum_{j=1}^{k}a_{j,n}%
\sum_{i=1}^{j}\left(  \prod_{m=i}^{j}\alpha_{n-m+1}\right)  ^{-1}%
v_{i}\nonumber
\end{gather}

The above quantity is independent of $u_{0}$ if and only if the coefficient of
$u_{0}$ is zero for all $n;$ i.e., if and only if $\{\alpha_{n}\}$ satisfies
the difference equation (\ref{ricn}). Dropping the $u_{0}$ terms leaves the
following:%
\begin{equation}
b_{n}-\sum_{j=1}^{k}a_{j,n}\left[  \sum_{i=1}^{j}\left(  \prod_{m=i}^{j}%
\alpha_{n-m+1}\right)  ^{-1}v_{i}\right]  =b_{n}-\sum_{j=1}^{k}\left[
\sum_{i=j}^{k}a_{i,n}\left(  \prod_{m=i}^{j}\alpha_{n-m+1}\right)
^{-1}\right]  v_{j}. \label{feq}%
\end{equation}

From this expression, we obtain the sc-factorization of (\ref{Lde}). The
cofactor equation (\ref{genlincf}) is simply (\ref{tdfslc}) while the factor
equation is obtained using (\ref{tdf1}), the above calculations and
(\ref{feq}). Finally, (\ref{genlinf}) is obtained by slightly adjusting the
summation indices to simplify notation.
\end{proof}

\begin{remark}
(The inversion form symmetry)\ Equation (\ref{ricn}) is not only a consequence
of the invertible map criterion but it is also related to \textit{a different}
semiconjugate factorization. Think of $\{\alpha_{n}\}$ as a solution of the
following $k$-th order difference equation on the (multiplicative) unit group
$G$ of $R$%
\begin{equation}
r_{n+1}=a_{0,n}+\sum_{j=1}^{k}a_{j,n}\left(  \prod_{i=0}^{j-1}r_{n-i}\right)
^{-1}.\label{homg}%
\end{equation}

Consider the homogeneous part of (\ref{Lde}) i.e., the linear equation
\begin{equation}
x_{n+1}=a_{0,n}x_{n}+a_{1,n}x_{n-1}+\cdots+a_{k,n}x_{n-k}. \label{homp}%
\end{equation}

If $a_{j,n}\in G$ for all $j$ then (\ref{homg}) turns out to be a factor
equation of a sc-factorization in $G$ of (\ref{homp}) and the cofactor
equation is $x_{n+1}=r_{n+1}x_{n};$ see \cite{sedh}. The \textit{inversion
form symmetry} that yields this sc-factorization is $[u_{0}u_{1}^{-1}%
,u_{1}u_{2}^{-1},\ldots,u_{k-1}u_{k}^{-1}]$. This (nonlinear) form symmetry is
characteristic of all difference equations that are \textit{homogeneous of
order one}, linear or not; see Chapter 4 in \cite{fsor}. Notice that the
factor equation (\ref{homg}) of (\ref{homp}) relative to the inversion form
symmetry is nonlinear whereas the factor equation (\ref{genlinf}) relative to
the linear form symmetry is linear. This makes a further reduction of order
more difficult for (\ref{homg}) which is also not homogeneous of order one.
\end{remark}

\section{Characteristic equation and eigensequences\label{cees}}

The sequence $\{\alpha_{n}\}$ in Theorem \ref{sf1} is a solution of the
nonlinear difference equation (\ref{homg}). Since $\{\alpha_{n}\}$ plays a
fundamental role in the sc-factorization of (\ref{Lde}), it is necessary to
examine (\ref{homg}) closely. This equation may be written in a way that does
not involve inversion. Multiply it on both sides by the quantity $r_{n}%
r_{n-1}\ldots r_{n-k+1}$\smallskip%

\begin{align*}
r_{n+1}r_{n}r_{n-1}\ldots r_{n-k+1}  &  =a_{0,n}(r_{n}r_{n-1}\ldots
r_{n-k+1})+a_{1,n}(r_{n-1}\ldots r_{n-k+1})+\\
&  \qquad+a_{2,n}(r_{n-2}\ldots r_{n-k+1})+\cdots+a_{k-1,n}r_{n-k+1}+a_{k,n}%
\end{align*}
which may be written more succinctly as%
\begin{equation}
\prod_{i=0}^{k}r_{n-i+1}-\sum_{j=0}^{k-1}a_{j,n}\left(  \prod_{i=j}%
^{k-1}r_{n-i}\right)  -a_{k,n}=0. \label{es}%
\end{equation}

This equation is not as esoteric as it may appear at first glance. To clarify,
consider the special homogeneous case with constant coefficients, i.e.,
\begin{equation}
x_{n+1}=a_{0}x_{n}+a_{1}x_{n-1}+\cdots+a_{k}x_{n-k}. \label{Ldeh}%
\end{equation}

Then (\ref{es}) reduces to the following difference equation%
\begin{equation}
r_{n+1}r_{n}\ldots r_{n-k+1}-a_{0}(r_{n}\ldots r_{n-k+1})-a_{1}(r_{n-2}\ldots
r_{n-k+1})-\cdots-a_{k-1}r_{n-k+1}-a_{k}=0. \label{rica}%
\end{equation}

A \textit{constant solution (or fixed point) }$r_{n}=r$ of (\ref{rica}) must
satisfy the polynomial equation%
\begin{equation}
r^{k+1}-a_{0}r^{k}-a_{1}r^{k-1}-\cdots a_{k}=0. \label{cpe}%
\end{equation}

The right hand side of (\ref{cpe}) is recognizable as the characteristic
polynomial of (\ref{Ldeh}) whose roots are indeed the eigenvalues of the
linear homogeneous equation (\ref{Ldeh}).

\begin{definition}
\label{eig}The difference equation (\ref{es}) in a ring $R$ is the
\textbf{characteristic equation} of the homogeneous part of (\ref{Lde}), i.e.,
the linear difference equation (\ref{homp}). Each solution of (\ref{es}) in
$R$ is an \textbf{eigensequence} of (\ref{homp}). An \textbf{eigenvalue }is a
constant eigensequence. An eigensequence whose every term is a unit in the
ring is \textbf{unitary}. An eigensequence containing a zero divisor is
\textbf{improper}.
\end{definition}

Equations (\ref{es}) and (\ref{homg}) are not equivalent: every solution of
(\ref{homg}) is a unitary solution of (\ref{es}) but nonunitary solutions may
also exist for (\ref{es}) that translate into eigensequences that are not
unitary, or not proper.

\begin{example}
\label{ez}The second-order linear difference equation%
\begin{equation}
x_{n+1}=2x_{n}-4x_{n-1},\quad x_{0},x_{1}\in\mathbb{Z} \label{exz}%
\end{equation}
has the characteristic equation $r_{n+1}r_{n}-2r_{n}+4=0$. This has no
constant solutions (eigenvalues) in $\mathbb{Z}$ since the polynomial
$r^{2}-2r+4$ has complex roots $r=1\pm i\sqrt{3}$. But this characteristic
equation does have a (nonunitary) period 3 solution given by
$\{1,-2,4,1,-2,4,\ldots\}$ as may be checked by direct substitution. This
eigensequence is trivially proper since $\mathbb{Z}$ has no zero divisors and
it is in fact unitary in the field $\mathbb{Q}$ of rational numbers.

The difference equation (\ref{exz}) is also valid in finite rings
$\mathbb{Z}_{m}$ of integers modulo a given positive integer $m$ and different
cases occur. For instance, $\mathbb{Z}_{17}$ is a field so the above
eigensequence is unitary. But in $\mathbb{Z}_{18}$ the same sequence of period
3 is an improper eigensequence since all even numbers are zero divisors.
Improper eigensequences are not desirable because ring extenstions do not
render them unitary and Theorem \ref{sf1} cannot be applied. Finally, we
mention that for some values of $m$ the polynomial $r^{2}-2r+4$ has roots in
$\mathbb{Z}_{m}$ which are eigenvalues of (\ref{exz}); e.g., in the field
$\mathbb{Z}_{5}$ the roots are 3 and 6 while in the ring $\mathbb{Z}_{12}$ the
roots are $-2\equiv10(\operatorname{mod}12)$ and $4$ both of which are improper.
\end{example}

\begin{example}
\label{ef}Consider the difference equation%
\begin{equation}
x_{n+1}=x_{n}+x_{n-1} \label{fib}%
\end{equation}
that is also known as the Fibonacci recurrence because with initial values
$x_{0}=0$ and $x_{1}=1$ (\ref{fib}) generates the Fibonacci sequence
1,1,2,3,5,8,13\ldots, denoted $\{F_{n}\}$.\ The characteristic equation of
(\ref{fib}) is%
\begin{equation}
r_{n+1}r_{n}-r_{n}-1=0. \label{cef}%
\end{equation}

This equation has no solutions in the ring of integers $\mathbb{Z}$, constant
or otherwise. For let $r_{1},r_{2}\in\mathbb{Z}$ and note that $r_{1}\not =0$
because clealry $r_{n}=0$ does not solve (\ref{cef}). Now (\ref{cef}) has a
solution $r_{2}\in\mathbb{Z}$ if and only if $r_{1}=\pm1.$ Either $r_{1}=1$,
$r_{2}=2$ so that $r_{3}=3/2\not \in \mathbb{Z}$ or $r_{1}=-1$, $r_{2}=0$ and
no value is defined for $r_{3}.$ Hence, (\ref{fib}) has no eigensequences in
$\mathbb{Z}$.

The eigenvalues (constant eigensequences) of this equation are roots
$(1\pm\sqrt{5})/2$ of its characteristic polynomial $r^{2}-r-1.$ Thus
(\ref{fib}) has no eigenvalues in $\mathbb{Q};$ but unlike $\mathbb{Z}$, in
$\mathbb{Q}$ (\ref{cef}) can be stated as $r_{n+1}=1+1/r_{n}$. Iteration
starting from (say) $r_{0}=1$ yields $r_{n}=F_{n+1}/F_{n}$, a unitary
eigensequence for (\ref{fib}). Theorem \ref{sf1} then yields a
sc-factorization, in $\mathbb{Q}$, of (\ref{fib}) consisting of the pair of
equations%
\begin{align*}
t_{n+1}  &  =a_{0,n}^{\prime}t_{n},\quad a_{0,n}^{\prime}=-\frac{F_{n}%
}{F_{n+1}}\\
x_{n+1}  &  =\frac{F_{n+2}}{F_{n+1}}x_{n}+t_{n+1}.
\end{align*}

\end{example}

We close this section with the following corollary of Theorem \ref{sf1} for
second order equations. See the next section for examples.

\begin{corollary}
\label{b0}Let $R$ be a ring with identity and assume that the difference
equation
\[
x_{n+1}=a_{0,n}x_{n}+a_{1,n}x_{n-1}%
\]

has a unitary eigensequence $\{r_{n}\}$ in $R$; i.e., $r_{n+1}=a_{0,n}%
+a_{1,n}r_{n}^{-1}$ for all $n\geq1.$ If there is an integer $m\geq1$ such
that $a_{1,m}=0$ then%
\[
x_{n}=r_{n}r_{n-1}\cdots r_{m+1}x_{m}\quad\text{for all }n\geq m+1.
\]

\end{corollary}

\begin{proof}
By Theorem \ref{sf1} the second-order difference equation has a
sc-factorization%
\begin{align*}
t_{n+1}  &  =a_{0,n}^{\prime}t_{n}=-a_{1,n}r_{n}^{-1}t_{n},\quad t_{1}%
=x_{1}-r_{1}x_{0}\\
x_{n+1}  &  =r_{n+1}x_{n}+t_{n+1}%
\end{align*}

If $a_{1,m}=0$ then $t_{m+1}=0.$ Thus, $t_{m+2}=-a_{m+1}r_{m+1}^{-1}t_{m+1}=0$
and by induction, $t_{n}=0$ for $n\geq m+1.$ This implies that the cofactor
equation reduces to $x_{n}=r_{n}x_{n-1}$ for $n\geq m+1$. Upon iteration we
obtain the formula for the solution $\{x_{n}\}$ of the second-order equation.
\end{proof}

\section{Periodic coefficients\label{sper}}

In this section we study the following difference equation with periodic
coefficients in a nontrivial ring $R$, i.e.,%
\begin{equation}
x_{n+1}=a_{n}x_{n}+b_{n}x_{n-1},\qquad a_{n+p_{1}}=a_{n},\ b_{n+p_{2}}%
=b_{n},\ n=0,1,2,\ldots\label{pco2}%
\end{equation}
where the (minimal or prime) periods $p_{1},p_{2}$ are positive integers with
least common multiple $p=\operatorname{lcm}(p_{1},p_{2})$; a central question
is whether (\ref{pco2}) has an eigensequence of period $p$ in $R$. If so then
such an eigensequence yields a sc-factorization of the second-order equation
into a pair of first-order ones. This may occur subject only to algebraic
restrictions on the ring concerning the existence of roots for a quadratic
polynomial but regardless of whether (\ref{pco2}) has any \textit{periodic
solutions}. In general, one cannot expect an equation such as (\ref{pco2}) to
possess any periodic solutions but the problem of existence of peridoic
solutions has been studied previously; see e.g., \cite{phi}.

An eigensequence of period $p$ exists in $R$ if there is an initial value
$r_{1}\in R$ such that the characteristic equation of (\ref{pco2}), i.e., the
first-order quadratic difference equation%
\begin{equation}
r_{n+1}r_{n}=a_{n}r_{n}+b_{n} \label{eper}%
\end{equation}
has a solution of period $p$ in the ring $R.$ Suppose that there are $r_{j}\in
R$ that satisfy (\ref{eper}) for $j=1,2,\ldots,p.$ Then%
\[
r_{2}r_{1}=a_{1}r_{1}+b_{1},\quad r_{3}r_{2}=a_{2}r_{2}+b_{2}%
\]

Let $L_{1}=r_{1}$ so that $r_{2}L_{1}=a_{1}L_{1}+b_{1}$. For $j=2,\ldots,p$
define $L_{j+1}=a_{j}L_{j}+b_{j}L_{j-1}.$ Then $r_{2}r_{1}=r_{2}L_{1}=L_{2}$
so that
\begin{align*}
r_{3}L_{2}  &  =(r_{3}r_{2})r_{1}=a_{2}r_{2}r_{1}+b_{2}r_{1}=a_{2}L_{2}%
+b_{2}L_{1}=L_{3},\\
r_{4}L_{3}  &  =(r_{4}r_{3})r_{2}r_{1}=a_{3}r_{3}r_{2}r_{1}+b_{3}r_{2}%
r_{1}=a_{3}L_{3}+b_{3}L_{2}=L_{4},\\
&  \vdots
\end{align*}

By induction, for $j=2,\ldots,p$
\[
r_{j+1}L_{j}=(r_{j+1}r_{j})r_{j-1}\cdots r_{1}=a_{j}r_{j}r_{j-1}\cdots
r_{1}+b_{3}r_{j-1}\cdots r_{1}=a_{j}L_{j}+b_{j}L_{j-1}=L_{j+1}.
\]

This process yields a solution $\{r_{n}\}$ of (\ref{eper}) with period $p$ if
and only if $r_{p+1}=r_{1};$ thus,%
\begin{equation}
r_{1}L_{p}=r_{p+1}L_{p}=a_{p}L_{p}+b_{p}L_{p-1}\Rightarrow(r_{1}-a_{p}%
)L_{p}=b_{p}L_{p-1} \label{Lr}%
\end{equation}

The quantities $L_{1},\ldots,L_{p}$ that are generated above evidently depend
on $r_{1}$ in a linear way so there are $\alpha_{j},\beta_{j}\in R$ such that
\[
L_{j}=\alpha_{j}r_{1}+\beta_{j}%
\]
for $j=1,2,\ldots,p.$ Inserting this form in (\ref{Lr}) yields%
\begin{align}
(r_{1}-a_{p})(\alpha_{p}r_{1}+\beta_{p})-b_{p}(\alpha_{p-1}r_{1}+\beta_{p-1})
&  =0\nonumber\\
r_{1}\alpha_{p}r_{1}+r_{1}\beta_{p}-(a_{p}\alpha_{p}+b_{p}\alpha_{p-1}%
)r_{1}-(a_{p}\beta_{p}+b_{p}\beta_{p-1})  &  =0 \label{p2q}%
\end{align}

The definition of $L_{j}$ implies%
\begin{align*}
\alpha_{j+1}r_{1}+\beta_{j+1}  &  =a_{j}(\alpha_{j}r_{1}+\beta_{j}%
)+b_{j}(\alpha_{j-1}r_{1}+\beta_{j-1})\\
&  =(a_{j}\alpha_{j}+b_{j}\alpha_{j-1})r_{1}+a_{j}\beta_{j}+b_{j}\beta_{j-1}%
\end{align*}

Suppose that $R$ has an identity 1. By matching coefficients on the two sides
of the above equality, we see that the coefficients $\alpha_{j},\beta_{j}$
satisfy (\ref{pco2}) for $j=1,2,\ldots,p$ with initial values
\begin{equation}
\alpha_{0}=0,\ \alpha_{1}=1;\quad\beta_{0}=1,\ \beta_{1}=0. \label{01}%
\end{equation}

Using this fact to simplify (\ref{p2q}) we conclude that if $r_{1}$ is a root
of the following polynomial
\begin{equation}
r\alpha_{p}r+r\beta_{p}-\alpha_{p+1}r-\beta_{p+1}=0 \label{p2q1}%
\end{equation}
then the solution $\{r_{n}\}$ of (\ref{eper}) has period $p.$ These
observations prove the following result.

\begin{theorem}
\label{per}Let $R$ be a ring with identity 1 and for $j=1,2,\ldots,p$, let
$\alpha_{j},\beta_{j}$ be obtained by iteration from (\ref{pco2}) subject to
(\ref{01}).

(a) If $r_{1}$ is a root of the quadratic polynomial (\ref{p2q1}) in $R$ and
there are $r_{j}\in R$ satisfying (\ref{eper}) for $j=2,\ldots,p$ then
$\{r_{n}\}_{n=1}^{\infty}$ is an eigensequence of (\ref{pco2}) with preiod $p$.

(b) If a root $r_{1}$ of (\ref{p2q1}) in $R$ is a unit and the recurrence%
\begin{equation}
r_{j+1}=a_{j}+b_{j}r_{j}^{-1} \label{epr}%
\end{equation}
\noindent generates units $r_{2},\ldots,r_{p}$ in $R$ then $\{r_{n}%
\}_{n=1}^{\infty}$ is a unitary eigensequence of (\ref{pco2}) with preiod $p$
that yields the sc-factorization%
\begin{align*}
t_{n+1}  &  =-b_{n}r_{n}^{-1}t_{n},\quad t_{1}=x_{1}-r_{1}x_{0}\\
x_{n+1}  &  =r_{n+1}x_{n}+t_{n+1}.
\end{align*}

\end{theorem}

The polynomial (\ref{p2q1}) simplifies further if the coefficients
$a_{j},b_{j}$ are in the center of $R$. Then $\alpha_{j},\beta_{j}$ are also
in the center of $R$ so (\ref{p2q1}) reduces to
\begin{equation}
\alpha_{p}r^{2}+(\beta_{p}-\alpha_{p+1})r-\beta_{p+1}=0. \label{p2qc}%
\end{equation}
\ 

Note that if $a_{n}=a$ and $b_{n}=b$ are constants then (with $p=1$) the
quadratic polynomial (\ref{p2qc}) takes the form $r^{2}-ar-b=0,$ i.e., the
characteristic polynomial of the autonomous linear equation of order 2.

\begin{example}
Consider the difference equation%
\begin{equation}
x_{n+1}=2\cos\left(  \frac{2\pi n}{3}\right)  x_{n}+x_{n-1} \label{c1}%
\end{equation}
where $a_{n}=\cos2\pi n/3$ has period 3 with $a_{1}=a_{2}=-1$ and $a_{3}=2$
and $b_{n}=1$ is constant (period 1). Let us assume that the underlying ring
is the field $\mathbb{R}$ of real numbers. The numbers $\alpha_{j},\beta_{j}$
are readily calculated from (\ref{c1}) using (\ref{01}):
\[
\alpha_{2}=-1,\ \alpha_{3}=2,\ \alpha_{4}=3,\ \beta_{2}=1,\ \beta
_{3}=-1,\ \beta_{4}=-1.
\]

The quadratic equation (\ref{p2qc}) with $p=3$ is $2r^{2}-4r+1=0$ in this
case. Of the two real roots $(2\pm\sqrt{2})/2$, let $r_{1}=(2-\sqrt{2})/2.$
Then using (\ref{epr}) we readily calculate $r_{2}=1+\sqrt{2}$, $r_{3}%
=-2+\sqrt{2}$. Since these are units in $\mathbb{R}$, by Theorem \ref{per} an
eigensequence with period 3 is obtained. If $\rho=-1/(r_{1}r_{2}r_{3})$ then
the sc-factorization of (\ref{c1}) is obtained by straightfoward calculation
with the factor equation%
\[
t_{3j+1}=\rho^{j}t_{1},\ t_{3j+2}=-\frac{\rho^{j}t_{1}}{r_{1}},\ t_{3j+3}%
=\frac{\rho^{j}t_{1}}{r_{1}r_{2}},\ j\geq0,\ t_{1}=x_{1}-r_{1}x_{0}%
\]
and the cofactor $x_{n+1}=r_{n+1}x_{n}+t_{n+1}.$ Since $\rho=1+\sqrt{2}$ all
solutions of (\ref{c1}) with $t_{1}\not =0$ are unbounded. However, $t_{1}=0$
when the initial values satisfy $x_{1}=r_{1}x_{0}$. Then $t_{n}=0$ for all $n$
and%
\[
x_{3n}=\frac{(-1)^{n}x_{0}}{\rho^{n}},\ x_{3n+1}=\frac{r_{1}(-1)^{n}x_{0}%
}{\rho^{n}},\ x_{3n+2}=\frac{r_{1}r_{2}(-1)^{n}x_{0}}{\rho^{n}},\ n\geq1.
\]

These special solutions of (\ref{c1}) converge to 0 exponentially for all
$x_{0}$.

The above calculations are meaningful in rings other than $\mathbb{R}$. By way
of comparison, now suppose that the underlying ring of (\ref{c1}) is a finite
field of type $\mathbb{Z}_{p}$ where $p\geq5$ is a prime. In particular, the
quadratic polynomial $2r^{2}-4r+1$ factors in $\mathbb{Z}_{7}$ with roots
$3,6$. If $r_{1}=3$ then using (\ref{epr}) we find that $r_{2}=4$ and
$r_{3}=1$ do in fact yield an eigensequence with period 3. Then, modulo 7,
$\rho=-1/(r_{1}r_{2}r_{3})=4$ and
\[
\rho^{3i}=\left(  \rho^{3}\right)  ^{i}=1^{i}=1,\quad\rho^{3i+1}=4,\quad
\rho^{3i+2}=2
\]
and this pattern of period 3 for distinct powers of $\rho$ yield a pattern of
period 9 for $t_{n}$ as shown in the following table%
\[%
\begin{tabular}
[c]{|c|c|c|c|}\hline
$j$ & 0 & 1 & 2\\\hline
$t_{3j+1}$ & $t_{1}$ & $4t_{1}$ & $2t_{1}$\\\hline
$t_{3j+2}$ & $2t_{1}$ & $t_{1}$ & $4t_{1}$\\\hline
$t_{3j+3}$ & $3t_{1}$ & $5t_{1}$ & $6t_{1}$\\\hline
\end{tabular}
\]
where $t_{1}=x_{1}-3x_{0}.$ In calculating the above entries we used
$1/r_{1}=5$ and $1/r_{1}r_{2}=3$ (modulo 7). Of course all solutions in
$\mathbb{Z}_{7}$ all $\{t_{n}\}$ and $\{x_{n}\}$ will be periodic with period
at most 48 since $\mathbb{Z}_{7}\times$ $\mathbb{Z}_{7}$ has that many points
(besides (0,0) which yields the trivial solution).
\end{example}

If quadratic polynomial (\ref{p2q1}) (or (\ref{p2qc}) in the commutative case)
has no roots in the underlying ring then periodic eigensequences with peirod
$p$ do not exist in that ring. However, other periodic eigensequences may
exist; e.g., for the autonomous equation (\ref{exz}) where $p=1$ recall that
there are no integer eigenvalues but there is an eigensequence of period 3 in
$\mathbb{Z}$. In some cases, a nonperiodic but still useful eigensequence may
exist as in the next example.

\begin{example}
Consider the following variant of (\ref{c1}) in $\mathbb{R}$:%
\begin{equation}
x_{n+1}=2\cos\left(  \frac{2\pi n}{3}\right)  x_{n}-x_{n-1}. \label{c2}%
\end{equation}

We find that $\alpha_{2}=-1,\ \alpha_{3}=0,\ \alpha_{4}=1,\ \beta
_{2}=-1,\ \beta_{3}=1,\ \beta_{4}=3.$ With these coefficients (\ref{p2qc}) has
no roots so (\ref{c2}) has no period-3 eigensequences. But the recurrence
(\ref{epr}) can be used to generate other types of eigensequences. For
instance, if $r_{1}=1$ then it can be verified by induction that%
\[
r_{3j+1}=3j+1,\ r_{3j+2}=-\frac{3j+2}{3j+1},\ r_{3j+3}=-\frac{1}{3j+2}%
,\ j\geq0
\]
is a (nonperiodic) eigensequence for (\ref{c2}). Note that $r_{n}%
r_{n+1}r_{n+2}=1$ for all $n$ so%
\[
t_{3j+1}=t_{1}\prod_{i=1}^{3j}\frac{1}{r_{i}}=t_{1}\prod_{i=0}^{j-1}\frac
{1}{r_{3i+1}r_{3i+2}r_{3i+3}}=t_{1}=x_{1}-x_{0}%
\]
and the factor equation may be expressed as%
\[
t_{3j+1}=t_{1},\ t_{3j+2}=\frac{t_{1}}{3j+1},\ t_{3j+3}=-\frac{t_{1}}%
{3j+2},\ j\geq0.
\]

The cofactor can now be specified as follows%
\begin{align*}
x_{3j+1}  &  =(3j+1)x_{3j}+t_{1},\\
x_{3j+2}  &  =-\frac{3j+2}{3j+1}x_{3j+1}+\frac{t_{1}}{3j+1}=-(3j+2)x_{3j}%
-t_{1},\\
x_{3j+3}  &  =-\frac{1}{3j+2}x_{3j+2}-\frac{t_{1}}{3j+2}=x_{3j}%
\end{align*}

The last equation implies that $x_{3j}=x_{0}$ for all $j$ so the general
solution of (\ref{c2}) is%
\[
x_{n}=\left\{
\begin{array}
[c]{l}%
x_{0}\quad\text{if }n=3j\\
nx_{0}+x_{1}\quad\text{if }n=3j+1\\
-nx_{0}-x_{1}\quad\text{if }n=3j+2
\end{array}
\right.  .
\]

In particular, if $x_{0}=0$ then the solution $\{0,x_{1},-x_{1},\ldots\}$ of
(\ref{c2}) has period 3 for all $x_{1}\not =0.$
\end{example}

\section{Unitary solutions and eigensequences\label{unisol}}

A potential difficulty in applying Theorem \ref{sf1} is finding the sequence
$\{\alpha_{n}\}$ that satisfies (\ref{homg}), i.e., finding a solution of
(\ref{es}). Even if such a solution exists then in many cases finding it
directly from (\ref{es}) is usually not easy. Fortunately, it is often
possible to calculate $\{\alpha_{n}\}$ \textit{indirectly}, by extracting it
from a \textit{unitary solution} of (\ref{homp}) in $R;$ i.e., a solution of
(\ref{homp}) that is contained in the unit group $G$. Let $\{x_{n}\}$ be such
a unitary solution for a given set of initial values $x_{0},x_{1},\ldots
,x_{k}\in G$. Multiplying (\ref{homp}) by $x_{n}^{-1}$ and rearranging terms
gives%
\begin{align*}
x_{n+1}x_{n}^{-1}  &  =a_{0,n}+a_{1,n}x_{n-1}x_{n}^{-1}+a_{2,n}x_{n-2}%
x_{n}^{-1}+\cdots+a_{k,n}x_{n-k}x_{n}^{-1}\\
&  =a_{0,n}+a_{1,n}x_{n-1}x_{n}^{-1}+a_{2,n}x_{n-2}(x_{n-1}^{-1}x_{n-1}%
)x_{n}^{-1}+\cdots\\
&  \quad\quad\quad+a_{k,n}x_{n-k}(x_{n-k+1}^{-1}x_{n-k+1})(x_{n-k+2}%
^{-1}x_{n-k+2})\cdots(x_{n-1}^{-1}x_{n-1})x_{n}^{-1}\\
&  =a_{0,n}+a_{1,n}(x_{n}x_{n-1}^{-1})^{-1}+a_{2,n}(x_{n-1}x_{n-2})^{-1}%
(x_{n}x_{n-1}^{-1})^{-1}+\cdots\\
&  \quad\qquad+a_{k,n}(x_{n-k+1}x_{n-k}^{-1})^{-1}(x_{n-k+2}x_{n-k+1}%
^{-1})^{-1}\cdots(x_{n}x_{n-1}^{-1})^{-1}%
\end{align*}

If $r_{n}=x_{n}x_{n-1}^{-1}$ for each $n$ then the above equation can be
written as%
\begin{align}
r_{n+1}  &  =a_{0,n}+a_{1,n}r_{n}^{-1}+a_{2,n}r_{n-1}^{-1}r_{n}^{-1}%
+\cdots+a_{k,n}r_{n-k+1}^{-1}r_{n-k+2}^{-1}\cdots r_{n-1}^{-1}r_{n}%
^{-1},\text{ or:}\nonumber\\
r_{n+1}  &  =a_{0,n}+a_{1,n}r_{n}^{-1}+a_{2,n}(r_{n}r_{n-1})^{-1}%
+\cdots+a_{k,n}(r_{n}r_{n-1}\ldots r_{n-k+1})^{-1} \label{ricn1}%
\end{align}
which is precisely Equation (\ref{homg}). Thus, the sequence $\{r_{n}\}$ of
right ratios of $\{x_{n}\}$ satisfies (\ref{homg}). It is often easier to find
a unitary solution of (\ref{homp}) than to look for a particular solution of
(\ref{homg}). Once a unitary solution of (\ref{homp}) is identified, an
eigensequence may be extracted from it using the next result that supplements
and completes Theorem \ref{sf1}.

\begin{theorem}
\label{er}Let $R$ be a ring with identity. A (unitary) sequence in $R$ is an
eigensequence of (\ref{homp}) if and only if it is the right ratio sequence of
a unitary solution of (\ref{homp}).
\end{theorem}

\begin{proof}
Let $\{r_{n}\}$ be a unitary eigensequence of (\ref{homp}), choose $x_{0}\in
G$ and define $x_{j}=r_{j}x_{j-1}$ for $j=1,\ldots,k.$ Then $x_{j}\in G$ for
each $j$ and%
\begin{align*}
r_{j+1}x_{j}  &  =(a_{0,n}+a_{1,n}r_{j}^{-1}+a_{2,n}r_{j-1}^{-1}r_{j}%
^{-1}+\cdots+a_{k,n}r_{j-k+1}^{-1}r_{j-k+2}^{-1}\cdots r_{j-1}^{-1}r_{j}%
^{-1})x_{j}\\
&  =a_{0,n}x_{j}+a_{1,n}r_{j}^{-1}x_{j}+a_{2,n}r_{j-1}^{-1}r_{j}^{-1}%
x_{j}+\cdots+a_{k,n}r_{j-k+1}^{-1}r_{j-k+2}^{-1}\cdots r_{j-1}^{-1}r_{j}%
^{-1}x_{j}\\
&  =a_{0,n}x_{j}+a_{1,n}x_{j-1}+a_{2,n}r_{j-1}^{-1}x_{j-1}+\cdots
+a_{k,n}r_{j-k+1}^{-1}r_{j-k+2}^{-1}\cdots r_{j-2}^{-1}r_{j-1}^{-1}x_{j-1}%
\end{align*}
where we used the fact that $r_{j}^{-1}x_{j}=x_{j-1}.$ Similarly,
$r_{j-1}^{-1}x_{j-1}=x_{j-2}$ which yields a further reduction%
\[
r_{j+1}x_{j}=a_{0,n}x_{j}+a_{1,n}x_{j-1}+a_{2,n}x_{j-2}+\cdots+a_{k,n}%
r_{j-k+1}^{-1}r_{j-k+2}^{-1}\cdots r_{j-2}^{-1}x_{j-2}.
\]

Next, $r_{j-2}^{-1}x_{j-2}=x_{j-3}$ and the above calculation may be continued
to ultimately yield%
\[
r_{j+1}x_{j}=a_{0,n}x_{j}+a_{1,n}x_{j-1}+a_{2,n}x_{j-2}+\cdots+a_{k,n}x_{j-k}%
\]

Define the right hand side as $x_{j+1}$, then proceed to $r_{j+2}x_{j+1}$ and
repeat the calculate to generate a new value%
\[
x_{j+2}=r_{j+2}x_{j+1}=a_{0,n}x_{j+1}+a_{1,n}x_{j}+a_{2,n}x_{j-1}%
+\cdots+a_{k,n}x_{j+1-k}%
\]

The values of $x_{n}$ generated by the above construction satisfy the linear
equation (\ref{homp}) for $n=j+1,j+2,\ldots$ Therefore, $\{x_{n}\}$ is a
unitary solution of (\ref{homp}) whose right ratio sequence is $\{r_{n}\}$ (by
construction). The converse is true by the definition of eigensequence and the
argument preceding this theorem.
\end{proof}

For the homogeneous difference equation (\ref{Ldeh}) with constant
coefficients the following is true.

\begin{corollary}
\label{ccde}Let $\{a_{i}\}$, $i=1,\ldots,k$ be constants in a ring $R$ with
identity such that $a_{k}\not =0$ and let $G$ be the unit group of $R.$

(a) Equation (\ref{Ldeh}) has a linear form symmetry if and only if it has a
unitary solution $\{u_{n}\}$ in $R.$ In this case, the corresponding
sc-factorization is determined by the eigensequence $\{\alpha_{n}\}$ of right
ratios $\alpha_{n}=u_{n}u_{n-1}^{-1}$ of $\{u_{n}\}$ as follows%
\begin{align}
t_{n+1}  &  =a_{0}^{\prime}t_{n}+a_{1}^{\prime}t_{n-1}+\cdots+a_{k-1}^{\prime
}t_{n-k+1},\label{ccf}\\
x_{n+1}  &  =\alpha_{n+1}x_{n}+t_{n+1},\nonumber
\end{align}

where $a_{m}^{\prime}=-\sum_{i=m+1}^{k}a_{i}\left(  \prod_{j=m+1}^{i}%
\alpha_{n-j+1}\right)  ^{-1}$ for $m=0,\ldots,k-1$.

(b) Every fixed point of (\ref{rica}) is an eigenvalue of (\ref{Ldeh}) in $R$.
Such a fixed point exists if and only if the characteristic polynomial
$r^{k+1}-a_{0}r^{k}-\cdots-a_{k-1}r-a_{k}$ has a root in $R.$
\end{corollary}

\begin{remark}
Equation (\ref{ccf}) is once again linear but with order less than
(\ref{Ldeh}) by one. If (\ref{ccf}) also possesses a unitary solution in $R$
then it is reducible in order by a second application of the preceding
corollary. This process may be repeated $k$ times to yield a triangular system
of $k+1$ first order linear equations; see \cite{fsor}. In particular, if $R$
is an algebraically closed field then the characteristic polynomial of
(\ref{Ldeh}) factors completely in $R$ and thus, a set of $k+1$ eigenvalues is
made available for a full sc-factorization.
\end{remark}

\begin{example}
\label{3ode}Consider the following third-order linear difference equation in
the finite field $\mathbb{Z}_{p}$ where $p$ is a prime number%
\begin{equation}
x_{n+1}=2x_{n-1}+x_{n-2},\quad x_{0},x_{1},x_{2}\in\mathbb{Z}_{p}.
\label{fib3}%
\end{equation}

The characteristic polynomial of (\ref{fib3}) is $r^{3}-2r-1=(r+1)(r^{2}-r-1)$
with a root $-1\in\mathbb{Z}_{p}$. Thus Corollary \ref{ccde} yields a
sc-factorization%
\[
t_{n+1}=a_{0}^{\prime}t_{n}+a_{1}^{\prime}t_{n-1},\quad x_{n+1}=-x_{n}%
+t_{n+1}.
\]
where $t_{1}=x_{0}+x_{1}$, $t_{2}=x_{1}+x_{2}$. The constant coefficients are
$a_{0}^{\prime}=a_{1}^{\prime}=1$ which yield the factor equation
$t_{n+1}=t_{n}+t_{n-1}$, i.e., the Fibonacci recurrence (\ref{fib}). Its
characteristic polynomial is $r^{2}-r-1$. Corollary \ref{ccde} may be applied
as long as (\ref{fib}) has a unitary solution in $\mathbb{Z}_{p}$, i.e., a
solution that never visits zero; $\mathbb{Z}_{p}$ contains such a
zero-avoiding solution of (\ref{fib}) for primes of type $p\equiv
0,1,4(\operatorname{mod}5)$ and infinitely many primes of type $p\equiv
2,3(\operatorname{mod}5)$; see \cite{gup} and \cite{sedf}. Let $p$ be such a
prime and $\{u_{n}\}$ a zero-avoiding solution of (\ref{fib}) in
$\mathbb{Z}_{p}.$ Repeating the calculations in Example \ref{ef} but replacing
$F_{n}$ with $u_{n}$ gives the following sc-factorization for (\ref{fib3})%
\begin{align*}
x_{n+1}  &  =-x_{n}+t_{n+1},\\
t_{n+1}  &  =\frac{u_{n+1}}{u_{n}}t_{n}+s_{n+1},\quad t_{1}=x_{1}+x_{0}\\
s_{n+1}  &  =-\frac{u_{n-1}}{u_{n}}s_{n},\qquad\quad\ s_{2}=t_{2}-\frac{u_{2}%
}{u_{1}}t_{1}.
\end{align*}

\end{example}

\section{A note on the Poincar\'{e}-Perron Theorem\label{ppt}}

The fact that eigensequences are ratio sequences of unitary solutions recalls
the celebrated theorem of Poincar\'{e} and Perron; see \cite{per} and
\cite{pon}, or Section 8.2 of \cite{elay}. In the language of eigensequences,
the theorem may be stated as follows: \medskip

\textquotedblleft\textit{Let }$R=\mathbb{C}$\textit{ and assume that the
coefficients }$a_{i,n}$\textit{ in (\ref{homp}) converge to constants }$a_{i}%
$\textit{ as }$n\rightarrow\infty$\textit{ ; i.e., (\ref{homp}) is a
Poincar\'{e} difference equation. Then each eigenvalue of (\ref{Ldeh}) is a
limit of an eigensequence of (\ref{homp}).\textquotedblright} \medskip

\begin{example}
Consider the following Poincar\'{e} equation in the field
$\mathbb{R}$ of real numbers%
\begin{equation}
x_{n+1}=\frac{1}{n}x_{n}+x_{n-1} \label{pp1}%
\end{equation}

The limiting autonomous equation for this is $y_{n+1}=y_{n-1}$ whose
eigenvalues are $\pm1,$ i.e., the roots of $r^{2}-1=0.$ The characteristic
equation of (\ref{pp1}) is
\begin{equation}
r_{n+1}=\frac{1}{n}+\frac{1}{r_{n}}. \label{pp1e}%
\end{equation}

It is readily verified by induction that the solution of (\ref{pp1e}) with
$r_{1}=1$ may be expressed as%
\[
r_{2n-1}=1,\quad r_{2n}=\frac{2n}{2n-1}.
\]

Thus $\lim_{n\rightarrow\infty}r_{n}=1$, as expected. Having the above eigensequence
explicitly is actually much more significant than this convergence result;
it yields the general solution of (\ref{pp1}). First, we obtain a
semiconjugate factorization of (\ref{pp1}) with factor equation%
\[
t_{2n}=-t_{2n-1},\quad t_{2n+1}=-\frac{2n-1}{2n}t_{2n}.
\]

By straightforward iteration%
\[
t_{2n+1}=\frac{(2n)!}{4^{n}(n!)^{2}},\quad t_{2n+2}=-t_{2n+1}.
\]

Finally, using the cofactor equation $x_{n+1}=r_{n+1}x_{n}+t_{n+1}$ a formula
for the general solution of (\ref{pp1}) may be obtained if desired.

\end{example}

Autonomous difference equations are trivially of Poincar\'{e}
type and each is its own limiting equation. Such equations are good for
illustrating some aspects of the Poincar\'{e}--Perron Theorem. For instance,
to see that not every eigensequence of a Poincar\'{e} difference equation
converges to an eigenvalue of the limiting equation, recall from earlier
discussion that Equation (\ref{exz}) has complex eigenvalues. Therefore, its
real eigensequences cannot converge to such eigenvalues.

Whether some eigensequences of a Poincar\'{e} difference equation in
topological rings more general than $\mathbb{R}$ or $\mathbb{C}$ (e.g., Banach
algebras)\ converge to eigenvalues of the limiting equation is an interesting
problem for future discussion.

\section{Zero-avoiding solutions and the ring of quotients\label{zav}}

If $R$ is a commutative ring with identity and no zero divisors (i.e., $R$ is
an \textit{integral domain}) then its complete ring of quotients is a field in
which $R$ is embedded (see, e.g., \cite{hun}, Chapter 3). We denote this
\textit{field of quotients} by $Q_{R}$ which contains an isomorphic copy of
$R$. The unit group of $Q_{R}$ is the set $Q_{R}\backslash\{0\}$ of all
nonzero elements of $Q_{R}$ which contains $R\backslash\{0\}$. If $Q_{R}$ is
not isomorphic to $R$ then $Q_{R}$ has an abundance of units that do not exist
in $R.$

Let us call a sequence $\{x_{n}\}$ \textit{zero-avoiding} if $x_{n}\not =0$
for all $n.$ In particular, if $R$ is a field then a sequence is zero-avoiding
if and only if it is unitary. The next result is a consequence of Theorems
\ref{sf1} and \ref{er} that reduces the search for eigensequences to a search
for zero-avoiding solutions.

\begin{theorem}
\label{qf}Let $R$ be an integral domain with field of quotients $Q_{R}$ and
assume that the parameters $a_{j,n},b_{n}$ in (\ref{Lde}) are in $R$ for all
$n$ and all $j=0,1,\ldots,k.$

(a) If $\{x_{n}\}$ is a zero-avoiding solution of (\ref{homp}) in $R$ then
$\{x_{n}x_{n-1}^{-1}\}$ is a (unitary) eigensequence of (\ref{homp}) in
$Q_{R}.$

(b) If (\ref{homp}) has a zero-avoiding solution $\{x_{n}\}$ in $R$ then
(\ref{Lde}) has a semiconjugate factorization in $Q_{R}$ consisting of the
pair of equations (\ref{genlinf}) and (\ref{genlincf}) with parameters
$\alpha_{n}=x_{n}x_{n-1}^{-1}$ and $a_{m,n}^{\prime}$ in $Q_{R}$ for all
$m,n.$
\end{theorem}

The earlier discussion of the Fibonacci recurrence (\ref{fib}) illustrates the
use of this theorem in a familiar case where the integral domain is
$\mathbb{Z}$ with field of quotients $\mathbb{Q}$.

\section{SC factorization without unitary solutions}

We now consider a difference equation whose characteristic equation
(\ref{rica}) always has a constant solution (eigenvalue) in an arbitrary ring.
This difference equation shows that the hypotheses in Theorems \ref{sf1} and
\ref{er} are \textit{not} necessary for the existence of a semiconjugate factorization.

\begin{theorem}
\label{o2}Let $R$ be an arbitrary nonzero ring, let $a,b$ be nonzero elements
in $R$ and let $\{c_{n}\}$ be an arbitrary sequence in $R.$ Consider the
second-order difference equation%
\begin{equation}
x_{n+1}=(a+b)x_{n}-abx_{n-1}+c_{n}. \label{gen2}%
\end{equation}

(a) The characteristic equation (\ref{rica}) has a constant solution $r=b$ and
(\ref{gen2}) has a semiconjugate factorization%
\begin{equation}
t_{n+1}=at_{n}+c_{n},\quad x_{n+1}=bx_{n}+t_{n+1}. \label{sco2}%
\end{equation}

(b) If $c_{n}=0$ for all $n$, i.e., (\ref{gen2}) is homogeneous, then every
solution $\{x_{n}\}$ of (\ref{gen2}) is given by the following where
$t_{1}=x_{1}-bx_{0}$:%
\begin{equation}
x_{n}=b^{n}x_{0}+\left(  a^{n-1}+b^{n-1}+\sum_{i=2}^{n-1}b^{n-i}%
a^{i-1}\right)  t_{1}. \label{so2}%
\end{equation}

\end{theorem}

\begin{proof}
(a) In this case the characteristic equation (\ref{rica}) is
\[
r_{n+1}r_{n}-(a+b)r_{n}+abr_{n-1}=0
\]

The constant solutions of this equation satisfy the polynomial equation
(\ref{cpe}), which is, in this case, $r^{2}-(a+b)r+ab=0$. This evidently has a
solution $r=b$ (if $R$ is commutative then $r=a$ is also a solution). Theorems
\ref{sf1} and \ref{er} are not applicable in the absence of an identity and
unitary solutions, but a sc-factorization of (\ref{gen2}) is readily obtained
by rearranging the terms of the equation as $x_{n+1}-bx_{n}=a(x_{n}%
-bx_{n-1})+c_{n}$ and defining a new variable $t_{n}=x_{n}-bx_{n-1}$ to obtain
the desired pair of equations (\ref{sco2}).

(b) To establish (\ref{so2}), let $c_{n}=0$ and use the first equation in
(\ref{sco2}) to obtain $t_{n}=a^{n-1}t_{1}$ with $t_{1}=x_{1}-bx_{0}.$ Now
iterate the equation $x_{n+1}=bx_{n}+a^{n}t_{1}$ and use $x_{1}=bx_{0}+t_{1}$
to obtain
\begin{align*}
x_{n}  &  =b^{n-1}x_{1}+\left(  \sum_{i=2}^{n-1}b^{n-i}a^{i-1}\right)
t_{1}+a^{n-1}t_{1}\\
&  =b^{n}x_{0}+\left(  a^{n-1}+b^{n-1}+\sum_{i=1}^{n}b^{n-i}a^{i-1}\right)
t_{1}%
\end{align*}
which is (\ref{so2}).
\end{proof}

\begin{example}
\label{bo}(Boolean rings) Let $R$ be a nonzero Boolean ring defined by the
relation $r^{2}=r$ for all $r\in R.$ It is straightforward to show that $R$ is
commutative and $2r=r+r=0$ for all $r$ (consider $(a+b)^{2}$ which must equal
to $a+b$). A concrete example of a Boolean ring is the collection of all
finite subsets of $\mathbb{Z}$ (including the empty set) under the operations
\[
A+B=(A\backslash B)\cup(B\backslash A),\quad AB=A\cap B
\]
for all finite $A,B\subset\mathbb{Z}.$ This is a commutative ring in which the
empty set is the zero element and every nonempty set is a zero divisor.

Theorem \ref{o2} applies in Boolean rings. From (\ref{so2}), if $c_{n}=0$ for
all $n$ then%
\begin{align*}
x_{n}  &  =bx_{0}+[a+b+(n-2)ba](x_{1}-bx_{0})\\
&  =\left\{
\begin{array}
[c]{l}%
bx_{0}+(a+b)(x_{1}-bx_{0}),\quad\quad\text{if }n\text{ is even}\\
bx_{0}+(a+b+ab)(x_{1}-bx_{0}),\text{ if }n\text{ is odd}%
\end{array}
\right. \\
&  =\left\{
\begin{array}
[c]{l}%
(a+b)x_{1}+abx_{0},\text{ if }n\text{ is even}\\
(a+b)x_{1}+abx_{1},\text{ if }n\text{ is odd}%
\end{array}
\right.
\end{align*}
where in the last step we used the fact that $-r=r$ for all $r$.
\end{example}

The next result shows how formula (\ref{so2}) takes a more familiar form in
the field of quotients of an integral domain.

\begin{corollary}
\label{bi}Let $R$ be an integral domain and let $a,b\in R.$ Then equation
(\ref{gen2}) has a semiconjugate factorization in $R$ given by (\ref{sco2})
and if $c_{n}=0$ then its solution (\ref{so2}) in $Q_{R}$ simplifies to the
classical formula:%
\begin{equation}
x_{n}=\left\{
\begin{array}
[c]{l}%
c_{1}a^{n}+c_{2}b^{n}\text{ if }a\not =b\\
\lbrack nx_{1}-(n-1)bx_{0}]b^{n-1}\text{ if }a=b
\end{array}
\right.  \label{evf}%
\end{equation}

where, using the reciprocal notation $1/u$ to denote the multiplicative
inverse\ $u^{-1}$ in $Q_{R},$
\[
c_{1}=\frac{x_{1}-bx_{0}}{a-b},\quad c_{2}=\frac{ax_{0}-x_{1}}{a-b}.
\]

\end{corollary}

\begin{proof}
The first assertion about a sc-factorization is an immediate consequence of
Theorem \ref{o2}. Formula (\ref{so2}) is defined in $R$ but it cannot be
simplified there. It does simplify in $Q_{R}$ with the aid of the geometric
sum formula as follows:%
\begin{align*}
x_{n}  &  =b^{n}x_{0}+b^{n-1}(x_{1}-bx_{0})\sum_{i=1}^{n}\left(  \frac{a}%
{b}\right)  ^{i-1}\quad(\text{defining }a^{0}=1,b^{0}=1)\\
&  =\left\{
\begin{array}
[c]{l}%
b^{n}x_{0}+b^{n-1}(x_{1}-bx_{0})[(a/b)^{n}-1][(a/b)-1]^{-1}\ \text{if }%
a\not =b\\
b^{n}x_{0}+b^{n-1}n(x_{1}-bx_{0})\text{ if}\ a=b
\end{array}
\right. \\
&  =\left\{
\begin{array}
[c]{l}%
c_{1}a^{n}+c_{2}b^{n}\text{ if }a\not =b\\
\lbrack nx_{1}-(n-1)bx_{0}]b^{n-1}\text{ if }a=b
\end{array}
\right.
\end{align*}
where $c_{1},c_{2}$ are defined as in the statement of the corollary.
\end{proof}

\begin{remark}
Note that in the case $a=b$ (\ref{evf}) is defined in $R$ itself. If $R$ is a
field then formula (\ref{evf}) holds in $R$ which is isomorphic to $Q_{R}$.
Care must be exercised in using this formula outside the context of fields.
For instance, Example \ref{bo} shows that (\ref{evf}) does not hold in a
Boolean ring (there is a mixed product term $ab$). Further, in rings of
functions that are discussed in the next section, $a,b$ are parametrized
quantities and thus they, or $a-b$ may fail to be units for some parameter values.
\end{remark}

\section{SC-Factorization in rings of functions\label{func}}

Difference equations in rings of functions often appear in applied
mathematics. Well-known special functions such as Bessel functions satisfy
recurrence relations that are examples of difference equations on rings of
real or complex-valued functions (see Examples \ref{mof} and \ref{cheb}
below). For functions from a nonempty set $S$ into a nonzero ring
$\mathcal{R}$ the operations of addition and multiplication are defined
\textit{pointwise}, i.e., for each $s\in S$
\[
(f+g)(s)=f(s)+g(s),\ (fg)(s)=f(s)g(s).
\]

Other types of ring operations are possible for functions but we consider only
the above pointwise operations. With these operations, the set $\mathcal{R}%
^{S}$ of all functions from $S$ into $\mathcal{R}$ is a function ring and each
subring $R(S)$ of $\mathcal{R}^{S}$ is a ring of $\mathcal{R}$-valued
functions on $S$. Note that $R(S)$ is commutative if $\mathcal{R}$ is. If
$R(S)$ contains all the constant functions on $S$ then we usually think of
these functions as elements of $\mathcal{R}$ and thus, think of $\mathcal{R}%
$\ as a subring of $R(S)$. \textit{In this section we assume that }%
$R(S)$\textit{ contains all the constants.}

A ring of functions $R(S)$ of the above type is also a function algebra; see,
e.g., \cite{hun} or \cite{mlb}. An element $u$ is a unit in $R(S)$ if and only
if $u(s)\not =0$ for all $s\in S.$\ In this case, the inverse of $u$ is just
its reciprocal $1/u$. Since $R(S)$ is closed under addition and
multiplication, if the parameters and initial values $a_{j,n},b_{n}%
,x_{j}:S\rightarrow\mathbb{R}$ are in $R(S)$ for all $j=0,1,\ldots k$ and all
$n$ then the solution $\{x_{n}\}$ of (\ref{Lde}) is also contained in $R(S)$.

In the familiar ring $C[0,1]$ of all continuous, real-valued functions on the
interval [0,1] the units are functions that are always either positive or
negative and a zero divisor is a function whose set of zeros has a nonempty
interior in [0,1]. The ring of polynomials $\mathcal{F}[x]$ with coefficients
in a given field $\mathcal{F}$ is a familiar ring that may be viewed as a ring
of functions on $\mathcal{F}$ (or some subset of it) by interpreting the
indeterminate as a variable; see \cite{mlb}, Chapter 4. In particular, if
$S=[0,1]$ and $\mathcal{F}=\mathbb{R}$ then by the Weierstrass approximation
theorem $\mathcal{F}[x]$ is a dense subring of $C[0,1]$ in the uniform
topology. Various rings of differentiable functions fall in-between
$\mathcal{F}[x]$ and the continuous functions and share the aforementioned
properties of the continuous functions. But larger rings such as bounded
functions or integrable functions have different properties.

\begin{corollary}
\label{cp}Let $R(S)$ be a ring of real-valued functions on a nonempty set $S.$
Assume that $a_{j,n}(s)\geq0$ for all $s\in S$, $j=0,1,\ldots,k$ and all $n.$
If%
\begin{equation}
\sum_{j=0}^{k}a_{j,n}(s)>0 \label{sn0}%
\end{equation}
for all $s\in S$ and all $n$ then the homogeneous part of the difference
equation (\ref{Lde}) has unitary solutions in $R(S).$ Therefore, (\ref{Lde})
has a sc-factorization in $R(S)$ that is given by (\ref{genlinf}) and
(\ref{genlincf}).
\end{corollary}

\begin{proof}
Let $a_{j,n}(s)\geq0$ for all $s\in S$ and all $n.$ Choose constant initial
values $u_{j}=1$ for $j=0,1,\ldots,k$ in (\ref{homp}), i.e., the homogeneous
part of (\ref{Lde}). By (\ref{sn0}), $\sum_{j=0}^{k}a_{j,n}(s)>0$ so\smallskip%

\[
u_{k+1}(s)=\sum_{j=0}^{k}a_{j,k}(s)>0
\]
for all $s\in S.$ Thus $u_{k+1}(s)$ is a unit in $R(S)$ and%
\[
u_{k+2}(s)=\sum_{j=0}^{k}a_{j,k+1}(s)u_{k+1-j}(s)=a_{0,k+1}(s)u_{k+1}%
(s)+\sum_{j=1}^{k}a_{j,k+1}(s)
\]
for all $s\in S.$ If $\sum_{j=1}^{k}a_{j,k+1}(s)=0$ for some $s$ then by
(\ref{sn0}) $a_{0,k+1}(s)\not =0.$ It follows that $u_{k+2}$ is also positive
on $S$, hence a unit in $R(S)$. Proceeding in this fashion, it follows that
$u_{n}(s)>0$ for all $s\in S$ and all $n.$ Thus $\{u_{n}(s)\}$ is a unitary
solution of (\ref{homp}). By Theorem \ref{er} the ratios sequence
$\{u_{n}(s)/u_{n-1}(s)\}$ is a unitary eigensequence in $R(S)$ so Theorem
\ref{sf1} yields a sc-factorization for (\ref{Lde}).
\end{proof}

\begin{remark}
\textit{Nonunitary} solutions for (\ref{homp}) exist under the hypotheses of
Corollary \ref{cp} because an initial function may not be a unit. Further,
\textit{none} of the parameters $a_{j,n}(s),b_{n}(s)$ in Corollary \ref{cp}
may be units. For instance, the corollary applies to the following difference
equation in $C[0,1]$%
\[
x_{n+1}(r)=a\left(  1-\sin n\pi r\right)  x_{n}(r)+br^{n}(1-r)x_{n-1}(r),\quad
a,b>0,\ r\in\lbrack0,1],\ n\geq1
\]
in which $a_{0,n}(r)=a\left(  1-\sin n\pi r\right)  $, $a_{1,n}(r)=br^{n}%
(1-r)$ and $b_{n}(r)=0$ are nonunits, but for all $r,n$%
\[
a_{0,n}(r)+a_{1,n}(r)=a\left(  1-\sin n\pi r\right)  +br^{n}(1-r)>0.
\]

\end{remark}

\begin{example}
\label{mof}(Modified Bessel functions) The second order, linear difference
equation%
\begin{equation}
x_{n+1}(s)=\frac{2n}{s}x_{n}(s)+x_{n-1}(s),\quad s\in(0,\infty) \label{mob}%
\end{equation}
is the recurrence relation for the modified Bessel functions $K_{n}(s)$ of the
second kind, so-called because they are solutions of the second-order linear
differential equation known as Bessel' s modified differential equation (see,
e.g., \cite{weis}). In fact, the sequence of functions $\{K_{n}(s)\}$ is a
particular solution of (\ref{mob}) from specified initial values
$K_{0}(s),K_{1}(s)$. According to Corollary \ref{cp} a unitary solution
$\{u_{n}(s)\}$ of (\ref{mob}) is generated by any pair of positive functions;
e.g., $u_{0}(s)=u_{1}(s)=1.$ The first few terms are
\[
u_{2}(s)=\frac{2}{s}+1,\ u_{3}(s)=\frac{8}{s^{2}}+\frac{4}{s}+1,\ u_{4}%
(s)=\frac{48}{s^{3}}+\frac{24}{s^{2}}+\frac{2}{s}+1
\]

Now the ratios $u_{n}(s)/u_{n-1}(s)$ define an eigensequence for (\ref{mob})
and yield the sc-factorization%
\[
t_{n+1}(s)=-\frac{u_{n-1}(s)}{u_{n}(s)}t_{n}(s)\quad x_{n+1}(s)=\frac
{u_{n+1}(s)}{u_{n}(s)}x_{n}(s)+t_{n+1}(s)
\]
with $t_{1}(s)=x_{1}(s)-[u_{1}(s)/u_{0}(s)]x_{0}(s)=x_{1}(s)-x_{0}(s).$
Iteration of the factor equation yields $t_{n}(s)=(-1)^{n-1}t_{1}%
(s)/u_{n-1}(s)$; inserting this into the cofactor, summation yields a formula
for the general solution of (\ref{mob}) in terms of the unitary solution
$\{u_{n}(s)\}$ as follows:
\begin{align*}
x_{n}(s)  &  =u_{n}(s)x_{1}(s)+\sum_{i=2}^{n-1}\frac{u_{n}(s)}{u_{i}(s)}%
t_{i}(s)\\
&  =u_{n}(s)\left[  x_{0}(s)+t_{1}(s)\sum_{i=1}^{n-1}\frac{(-1)^{i-1}}%
{u_{i}(s)u_{i-1}(s)}\right]  .
\end{align*}

Different values of positive functions $u_{0}(s),u_{1}(s)$ yield different
formulas but of course, the same quantity $x_{n}(s).$
\end{example}

In the next result, the coefficients are sequences of \textit{constant}
functions; however, they may exist in any nontrivial field, not just the real numbers.

\begin{corollary}
\label{fcc}Let $\mathcal{F}$ be a nonzero field, $S$ a nonempty set and
$\mathcal{F}(S)$ a ring of functions from $S$ into $\mathcal{F}$ that contains
all the constants in $\mathcal{F}$. Consider the following difference equation
in $\mathcal{F}(S)$%
\begin{equation}
x_{n+1}(s)=a_{0,n}x_{n}(s)+a_{1,n}x_{n-1}(s)+\cdots+a_{k,n}x_{n-k}(s)+b_{n}(s)
\label{num}%
\end{equation}
where the coefficients $a_{j,n}$ are constants, i.e., $a_{j,n}\in\mathcal{F}$
for $j=0,1,\ldots k$ and $b_{n}\in\mathcal{F}(S)$ for every $n.$ If the
homogeneous part of (\ref{num}) has a zero-avoiding solution in $\mathcal{F}$
(as a subset of $\mathcal{F}(S)$) then (\ref{num}) has a semiconjugate
factorization in $\mathcal{F}(S).$
\end{corollary}

\begin{proof}
Let $y_{0},y_{1}$ be constant functions, which may be thought of as elements
of $\mathcal{F}$ such that the solution $\{y_{n}\}$ of the homogeneous part of
(\ref{num}) that they generate is zero-avoiding. Then $y_{n}\not =0$ for all
$n$ so that the ratios $\alpha_{n}=y_{n}/y_{n-1}$ are well-defined units in
$\mathcal{F}(S)$. Thus by Theorem \ref{er} we have a unitary eigensequence and
Theorem \ref{sf1} yields a sc-factorization for (\ref{num}).
\end{proof}

Going in a different direction than Corollary \ref{fcc} we use Theorem
\ref{o2} to obtain semiconjugate factorizations and solutions of certain
second-order difference equations in $R(S)$ to which the preceding results may
not apply.

\begin{corollary}
\label{c01}Let $R(S)$ denote a ring of complex-valued functions on a nonempty
set $S$ and let $f,g,w_{n}\in R(S)$ for all $n\geq1.$ The linear difference
equation%
\begin{equation}
x_{n+1}(s)=f(s)x_{n}(s)+g(s)x_{n-1}(s)+w_{n}(s) \label{o2c}%
\end{equation}

has the semiconjugate factorization%
\[
t_{n+1}(s)=\frac{f(s)+h(s)}{2}t_{n}(s)+w_{n}(s),\quad x_{n+1}(s)=\frac
{f(s)-h(s)}{2}x_{n}(s)+t_{n+1}(s)
\]

in $R(S)$ where $h(s)=[f^{2}(s)+4g(s)]^{1/2}$.
\end{corollary}

\begin{proof}
Let $a(s)=[f(s)+h(s)]/2$ and $b(s)=[f(s)-h(s)]/2$. Then $f(s)=a(s)+b(s)$ and
$g(s)=-a(s)b(s)$ so (\ref{o2c}) can be written in the form (\ref{gen2}) and
Theorem \ref{o2} completes the proof.
\end{proof}

We now apply the ideas in Corollary \ref{c01} and Corollary \ref{bi} to obtain
general solutions of linear, second-order difference equations in rings of
functions. Assume that $f(s)$ and $g(s)$ are \textit{real-valued} functions in
the ring $R(S)$ of Corollary \ref{c01} and consider the homogeneous difference
equation%
\begin{equation}
x_{n+1}(s)=f(s)x_{n}(s)+g(s)x_{n-1}(s). \label{hlo2}%
\end{equation}

Let%
\[
S_{+}=\{s\in S:h^{2}(s)>0\},\ S_{-}=\{s\in S:h^{2}(s)<0\},\ S_{0}=\{s\in
S:h^{2}(s)=0\}.
\]

$S=S_{+}\cup S_{-}\cup S_{0}$ is a disjoint union of the above three sets, at
least one of which is nonempty. For $s\in S_{+}\cup S_{0}$ the formulas in
Corollary \ref{bi} yield correct solutions within $R(S)$ to the system of
equations in Corollary \ref{c01}.

If $S_{-}$ is nonempty and $s\in S_{-}$ then $g(s)<0$ and
\[
h(s)=i\sqrt{-f^{2}(s)-4g(s)},\quad a(s)=\frac{1}{2}\left[  f(s)+i\sqrt
{-f^{2}(s)-4g(s)}\right]  ,\quad b(s)=\overline{a(s)}%
\]
where the definition of $a(s)$ comes from the proof of Corollary \ref{c01}.
Switching to polar coordinates, the modulus $\left\vert a(s)\right\vert
=\left\vert b(s)\right\vert =\sqrt{-g(s)}$ and%
\[
a(s)=\sqrt{-g(s)}[\cos\theta(s)+i\sin\theta(s)],\quad b(s)=\sqrt{-g(s)}%
[\cos\theta(s)-i\sin\theta(s)]
\]
where
\[
\theta(s)=\cos^{-1}\frac{f(s)}{2\sqrt{-g(s)}}.
\]

Repeating the calculations in the proof of Corollary \ref{bi} applied to the
sc-factorization in Corollary \ref{c01} readily yields%
\[
x_{n}(s)=\frac{x_{1}(s)-\overline{a(s)}x_{0}(s)}{a(s)-\overline{a(s)}}%
a^{n}(s)+\frac{a(s)x_{0}(s)-x_{1}(s)}{a(s)-\overline{a(s)}}\left[
\overline{a(s)}\right]  ^{n}.
\]

Dropping the parentheses to reduce notational clutter and changing to polar
coordinates, straightforward calculation yields%
\begin{equation}
x_{n}=\left(  -g\right)  ^{n/2}\left[  x_{0}\cos\left(  n\cos^{-1}\frac
{f}{2\sqrt{-g}}\right)  +\frac{fx_{0}-2x_{1}}{f\,^{2}+4g}\sin\left(
n\cos^{-1}\frac{f}{2\sqrt{-g}}\right)  \right]  . \label{cxf}%
\end{equation}

Note that if $x_{0},x_{1}$ are real valued functions then so is $x_{n}$ for
every $n$.

\begin{example}
\label{cheb}\textbf{(}Chebyshev polynomials\textbf{) }Consider the linear
difference equation%
\begin{equation}
x_{n+1}(s)=2sx_{n}(s)-x_{n-1}(s) \label{eq0}%
\end{equation}
which is also the recurrence relation for the so-called Chebyshev polynomials
$T_{n}(s)$ with the initial functions $T_{0}(s)=1$ and $T_{1}(s)=s$ (for each
$n$, $T_{n}(s)$ is a solution of Chebyshev's second-order linear differential
equation; see, e.g.,\cite{weis}). If $f(s)=2s$ and $g(s)=-1$ then
$S_{-}=(-1,1)$, $S_{0}=\{-1,1\}$ and $S_{+}=\mathbb{R}\backslash\lbrack-1,1].$
For $s\in(-1,1),$%
\[
h(s)=2i\sqrt{1-s^{2}},\quad a(s)=s+i\sqrt{1-s^{2}}%
\]
so the formula for the general solution of (\ref{eq0}) is obtained from
(\ref{cxf}) as%
\begin{equation}
x_{n}(s)=\left[  x_{0}(s)\cos\left(  n\cos^{-1}s\right)  +\frac{sx_{0}%
(s)-x_{1}(s)}{s^{2}-1}\sin\left(  n\cos^{-1}s\right)  \right]  ,\ s\in(-1,1).
\label{cbf}%
\end{equation}

In particular, for $n\geq1$,
\begin{equation}
T_{n}(s)=\cos\left(  n\cos^{-1}s\right)  \qquad-1<s<1 \label{cch}%
\end{equation}

\noindent which are commonly known to be polynomials. If $s\in S_{+}$ then
repeating the calculations in Corollary \ref{bi} yields%
\begin{align*}
x_{n}(s)  &  =\frac{x_{1}(s)-b(s)x_{0}(s)}{a(s)-b(s)}a^{n}(s)+\frac
{a(s)x_{0}(s)-x_{1}(s)}{a(s)-b(s)}b^{n}(s)\\
&  =\frac{x_{1}(s)-(s-\sqrt{s^{2}-1})x_{0}(s)}{2\sqrt{s^{2}-1}}\left(
s+\sqrt{s^{2}-1}\right)  ^{n}+\\
&  \qquad+\frac{(s+\sqrt{s^{2}-1})x_{0}(s)-x_{1}(s)}{2\sqrt{s^{2}-1}}\left(
s-\sqrt{s^{2}-1}\right)  ^{n}.
\end{align*}

In particular,%
\begin{equation}
T_{n}(s)=\frac{1}{2}\left(  s+\sqrt{s^{2}-1}\right)  ^{n}+\frac{1}{2}\left(
s-\sqrt{s^{2}-1}\right)  ^{n},\ s\in\mathbb{R}\backslash\lbrack-1,1].
\label{rch}%
\end{equation}

Finally, if $s\in S_{0}$ i.e., $s=\pm1$ then $a(s)=b(s)=s$ so again from
Corollary \ref{bi}%
\begin{align*}
x_{n}(s)  &  =[nx_{1}(s)-(n-1)b(s)x_{0}(s)]b^{n-1}(s)\\
&  =\left\{
\begin{array}
[c]{l}%
nx_{1}(1)-(n-1)x_{0}(1),\ \text{if }s=1\\
\lbrack nx_{1}(-1)+(n-1)x_{0}(-1)](-1)^{n},\ \text{if }s=-1
\end{array}
\right.  .
\end{align*}

These relations yield $T_{n}(1)=1$ and $T_{n}(-1)=(-1)^{n}$ for all $n;$ it
follows that formula (\ref{rch}) also holds for $s=\pm1.$ Though not
immediately apparent, expressions $T_{n}$ defined by (\ref{cch}) and
(\ref{rch}) taken together yield polynomials on the real line.
\end{example}

\section{Summary and future directions}

In this paper we studied semiconjugate factorizations of linear difference
equations in rings. For the typical (recursive) linear equation this method
supplements the standard methods that use modules and group actions.

In some cases, we obtained several sufficient conditions that are not
necessary for the existence of sc-factorizations. Relaxing or modifying the
hypotheses in these cases should lead to broader applicability. For example,
extending Corollary \ref{cp} to include negative coefficients or extending
Corollary \ref{c01} to include dependence on $n$ will lead, among other
things, to results such as Examples \ref{mof} and \ref{cheb} on recurrences of
many more special functions such as Bessel, Legendre, Hermite, etc.

Many challenges at different levels of generality remain. A ring with identity
may contain no unitary solutions of the homogeneous part of (\ref{Lde}). For
example, it is shown in \cite{sedf} that for certain primes, e.g.,
$2,3,7,23,\ldots$ the field $\mathbb{Z}_{p}$ contains no zero-avoiding (hence
unitary) solutions of (\ref{fib}). In such cases (\ref{Lde}) fails to have the
linear form symmetry (\ref{lfs}). However, nonexistence of a particular form
symmetry is not equivalent to the nonexistence of a sc-factorization; see the
remarks following the proof of Theorem \ref{sf1} about the (nonlinear)
inversion form symmetry. The question of whether in the absence of (\ref{lfs})
other types of form symmetry exist for which the factor or the cofactor
equation is linear, remains open.

If $R$ is not commutative then using different orders of multiplications of
coefficients $a_{j,n}$ with the variables $x_{n-j}$ in (\ref{Lde}) may result
in different equations (with the same set of coefficients). Not all the
results in this paper extend readily to these variants in the noncommutative cases.

Finally, linear difference equations may occur in nonrecursive forms; e.g.,
$a_{0,n}x_{n}+a_{1,n}x_{n-1}+\cdots+a_{k,n}x_{n-k}=b_{n}$ where the leading
coefficient $a_{0,n}$ is not a unit in the ring $R$ for infinitely many $n$
and the equation cannot be solved uniquely for $x_{n}$. Nonrecursive,
nonlinear difference equations may occur in conjunction with
\textit{recursive} linear equations; indeed, the characteristic equations
(\ref{es}) and (\ref{rica}) are nonrecursive polynomial equations. For
nonrecursive difference equations even such basic issues as the existence and
uniqueness of solutions is not assured. To illustrate, consider the linear
homogeneous equation
\begin{equation}
4x_{n+1}+6x_{n}+2x_{n-1}=0 \label{nrec}%
\end{equation}
in the finite ring $\mathbb{Z}_{8}$. Since 4 is a zero divisor in
$\mathbb{Z}_{8}$ the above equation is not recursive. Hence it does not unfold
to a map of $\mathbb{Z}_{8}\times\mathbb{Z}_{8}$ in the standard way and
semiconjugate factorization is not applicable. However, (\ref{nrec}) does
split into a triangular, factor/cofactor pair of equations by a simple
rearrangement of terms: $4(x_{n+1}+x_{n})+2(x_{n}+x_{n-1})=0$ which by
defining the new variable $t_{n}=x_{n}+x_{n-1}$ yields%
\begin{align}
4t_{n+1}+2t_{n}  &  =0,\label{nfac}\\
x_{n+1}  &  =-x_{n}+t_{n+1}. \label{cnfac}%
\end{align}

It is evident that every sequence $\{t_{n}\}$ in $\{0,4\}^{\mathbb{N}}$ is a
solution of the (nonrecursive) first-order factor equation (\ref{nfac}) in
$\mathbb{Z}_{8}.$ The (recursive) cofactor equation (\ref{cnfac}) then
generates a unique solution of (\ref{nrec}) in $\mathbb{Z}_{8}$ for each
solution of (\ref{nfac}); the table below lists the first few terms of a
sequence $\{t_{n}\}$ and the corresponding terms of the solution of
(\ref{nrec}) with initial values $x_{0}=1$, $x_{1}=3$ and $t_{1}=x_{1}%
+x_{0}=4$

\begin{center}%
\begin{tabular}
[c]{|c|c|c|c|c|c|c|c|c|c|c|c|c|c|c|c|c|}\hline
$n$ & 1 & 2 & 3 & 4 & 5 & 6 & 7 & 8 & 9 & 10 & 11 & 12 & 13 & 14 & 15 &
$\cdots$\\\hline
$t_{n}$ & 4 & 4 & 4 & 4 & 4 & 0 & 0 & 0 & 0 & 0 & 4 & 4 & 4 & 4 & 4 & $\cdots
$\\\hline
$x_{n}$ & 3 & 1 & 3 & 1 & 3 & 5 & 3 & 5 & 3 & 5 & 7 & 5 & 7 & 5 & 7 & $\cdots
$\\\hline
\end{tabular}

\end{center}

The same pair of initial values corresponds to infinitely many solutions of
(\ref{nrec}) depending on the choice of $\{t_{n}\};$ e.g.,

\begin{center}%
\begin{tabular}
[c]{|c|c|c|c|c|c|c|c|c|c|c|c|c|c|c|c|c|}\hline
$n$ & 1 & 2 & 3 & 4 & 5 & 6 & 7 & 8 & 9 & 10 & 11 & 12 & 13 & 14 & 15 &
$\cdots$\\\hline
$t_{n}$ & 4 & 0 & 4 & 0 & 0 & 4 & 0 & 0 & 0 & 4 & 0 & 0 & 0 & 0 & 4 & $\cdots
$\\\hline
$x_{n}$ & 3 & 5 & 7 & 1 & 7 & 5 & 3 & 5 & 3 & 1 & 7 & 1 & 7 & 1 & 3 & $\cdots
$\\\hline
\end{tabular}

\end{center}

By contrast, the situation is quite different in $\mathbb{Z}_{9}$ because 4 is
a unit and its reciprocal modulo 9 is 7. Thus upon dividing by 4, (\ref{nrec})
is reduced to the recursive linear equation $x_{n+1}+6x_{n}+5x_{n-1}=0$ or
equivalently, $x_{n+1}=3x_{n}+4x_{n-1}$ to which the methods discussed in this
paper are applicable (though in this simple case the sc-factorization may be
quickly obtained via the rearrangement of terms mentioned above). These
observations indicate that nonrecursive linear difference equations and the
challenging problems associated with them may turn out to be sources of
profound new ideas.

\end{document}